\documentclass{amsart}
\usepackage{amsmath,amsthm}
\usepackage{amsfonts,amssymb}
\usepackage{enumerate}
\usepackage{srcltx}

\newtheorem{theorem}{Theorem}[section]
\newtheorem{corollary}[theorem]{Corollary}

\numberwithin{equation}{section}
\theoremstyle{remark}

\def\({\left(}
\def \){ \right)}

\def\S{\mathbb{S}}

\newcommand{\wh}{\widehat}

\newcommand\calF{F}
\newcommand\calH{H}

\renewcommand\>{\rangle}

\begin{document}

\title[Pitt inequality and logarithmic uncertainty principle]
{Sharp Pitt inequality and logarithmic uncertainty principle for Dunkl
transform in $L^{2}$}

\author{D.~V.~Gorbachev}
\address{D.~Gorbachev, Tula State University,
Department of Applied Mathematics and Computer Science,
300012 Tula, Russia}
\email{dvgmail@mail.ru}

\author{V.~I.~Ivanov}
\address{V.~Ivanov, Tula State University,
Department of Applied Mathematics and Computer Science,
300012 Tula, Russia}
\email{ivaleryi@mail.ru}

\author{S.~Yu.~Tikhonov}
\address{S. Tikhonov, ICREA, Centre de Recerca Matem\`{a}tica, and UAB\\
Campus de Bellaterra, Edifici~C
08193 Bellaterra (Barcelona), Spain.}
\email{stikhonov@crm.cat}

\date{\today}

\thanks{The first and the second authors were partially supported by RFFI N\,13-01-00045, Ministry of education and science of Russian Federation
(N\,5414GZ, N\,1.1333.2014K), and Dmitry Zimin's Dynasty Foundation.
The third author was partially supported by MTM 2011-27637, 2014 SGR 289, and RFFI 13-01-00043.}

\begin{abstract}
We prove sharp Pitt's inequality for the Dunkl transform in $L^{2}(\mathbb{R}^{d})$
with the corresponding weights. As an application, we obtain the logarithmic
uncertainty principle for the Dunkl transform.
\end{abstract}

\maketitle

\section{Introduction}

Let $\Gamma(t)$ be the gamma function, $\mathbb{R}^{d}$ be the real space of
$d$ dimensions, equipped with a scalar product $\<x,y\>$ and a norm
$|x|=\sqrt{\<x,x\>}$. Denote by $\mathcal{S}(\mathbb{R}^{d})$ the Schwartz
space on $\mathbb{R}^{d}$ and by $L^{2}(\mathbb{R}^{d})$ the Hilbert space of
complex-valued functions endowed with a norm
$\|f\|_{2}=\left(\int_{\mathbb{R}^{d}}|f(x)|^{2}\,dx\right)^{1/2}$. The Fourier
transform is defined~by
\[
\widehat{f}(y)=(2\pi)^{-d/2}\int_{\mathbb{R}^{d}}f(x)e^{-i\<x,y\>}\,dx.
\]

W.~Beckner~\cite{Bec} proved the Pitt inequality for the Fourier transform
\begin{equation}\label{eq1.1}
\bigl\||y|^{-\beta}\widehat{f}(y)\bigr\|_{2}\le
C(\beta)\bigl\||x|^{\beta}f(x)\bigr\|_{2},\quad f\in
\mathcal{S}(\mathbb{R}^{d}),\qquad 0<\beta<d/2,
\end{equation}
with sharp constant
\[
C(\beta)=2^{-\beta}\,\frac{\Gamma(\frac{1}{2}(\frac{d}{2}-\beta))}{\Gamma(\frac{1}{2}(\frac{d}{2}+\beta))}.
\]
Noting that
$\big\|(-\Delta)^{\beta/2}f(x)\big\|_{2}=\big\||x|^{\beta}\wh{f}(x)\big\|_{2}$, Pitt's
inequality can be viewed as a Hardy--Rellich inequality
$\big\||x|^{-\beta}f(x)\big\|_{2}\le C(\beta)\big\|(-\Delta)^{\beta/2}f(x)\big\|_{2}$; see the papers by D.~Yafaev~\cite{Yaf1} and
S.~Eilertsen~\cite{Eil} for alternative proofs and extensions of \eqref{eq1.1}.

For $\beta=0$, \eqref{eq1.1} is the Plancherel theorem. If $\beta>0$ there is
no extremiser in inequality \eqref{eq1.1} and its sharpness can be obtained on
the set of radial functions.

The proof of \eqref{eq1.1} in~\cite{Bec} is based on an equivalent integral
realization as a Stein--Weiss fractional integral on $\mathbb{R}^d$. D.~Yafaev
in \cite{Yaf1} used the following decomposition (see \cite[Chapt.~IV]{SteWei})
\begin{equation}\label{eq1.2}
L^{2}(\mathbb{R}^{d})=\sum_{n=0}^{\infty}\oplus \mathcal{R}_{n}^{d},
\end{equation}
where $\mathcal{R}_{0}^{d}$ is the space of radial function, and
$\mathcal{R}_{n}^{d}=\mathcal{R}_{0}^{d}\otimes \mathcal{H}_{n}^{d}$ is the
space of functions in $\mathbb{R}^d$ that are products of radial functions and
spherical harmonics of degree $n$. Thanks to this decomposition it is enough to
study inequality \eqref{eq1.1} on the subsets of $\mathcal{R}_{n}^{d}$ which
are invariant under the Fourier transform.

In this paper, following \cite{Yaf1} and using similar decomposition of the
space $L^{2}(\mathbb{R}^{d})$ with the Dunkl weight, we prove sharp Pitt's
inequality for the Dunkl transform.

Let $R\subset \mathbb{R}^{d}$ be a root system, $R_{+}$ be the positive
subsystem of $R$, and $k\colon R\to \mathbb{R}_{+}$ be a multiplicity function
with the property that $k$ is $G$-invariant. Here $G(R)\subset O(d)$ is a
finite reflection group generated by reflections $\{\sigma_{a}\colon a\in R\}$,
where $\sigma_{a}$ is a reflection with respect to a hyperplane $\<a,x\>=0$.

Throughout this paper we let
\[
v_{k}(x)=\prod_{a\in R_{+}}|\<a,x\>|^{2k(a)}
\]
denote the Dunkl weight. Moreover, $d\mu_{k}(x)=c_{k}v_{k}(x)\,dx$, where
\[
c_{k}^{-1}=\int_{\mathbb{R}^{d}}e^{-|x|^{2}/2}v_{k}(x)\,dx
\]
is the Macdonald--Mehta--Selberg integral. Let $L^{2}(\mathbb{R}^{d},d\mu_{k})$
be the Hilbert space of complex-valued functions endowed with a norm
$\|f\|_{2,d\mu_{k}}=\left(\int_{\mathbb{R}^{d}}|f(x)|^{2}\,d\mu_{k}(x)\right)^{1/2}$.

Introduced by C.~F.~Dunkl, a family of differential-difference operators
(Dunkl's operators) associated with $G$ and $k$ are given by
\[
D_{j}f(x)=\frac{\partial f(x)}{\partial x_{j}}+ \sum_{a\in
R_{+}}k(a)\<a,e_{j}\>\,\frac{f(x)-f(\sigma_{a}x)}{\<a,x\>},\quad j=1,\ldots,d,
\]
The Dunkl kernel $e_{k}(x, y)=E_{k}(x, iy)$ is the unique solution of
the joint eigenvalue problem for the corresponding Dunkl operators:
\[
D_{j}f(x)=iy_{j}f(x),\quad j=1,\ldots,d,\qquad f(0)=1.
\]
Let us define the Dunkl transforms as follows
\[
\calF_{k}(f)(y)=\int_{\mathbb{R}^{d}}f(x)\overline{e_{k}(x,
y)}\,d\mu_{k}(x), \quad \calF_{k}^{-1}(f)(x)=\calF_{k}(f)(-x),
\]
where $\calF_{k}(f)$ and $\calF_{k}^{-1}(f)$ are the direct and inverse
transforms correspondingly (see, e.g., \cite{Ros}). For $k\equiv 0$ we have
$\calF_{0}(f)=\widehat{f}$.

Our goal is to study Pitt's inequality for the Dunkl transform
\begin{equation}\label{eq1.3}
\||y|^{-\beta}\calF_{k}(f)(y)\|_{2,d\mu_{k}}\le
C(\beta,k)\||x|^{\beta}f(x)\|_{2,d\mu_{k}},\quad f\in \mathcal{S}(\mathbb{R}^{d}),
\end{equation}
with the sharp constant $C(\beta,k)$.

Let us first recall some known results on Pitt's inequality for the Hankel transform.
Let $\lambda\ge -1/2$. Denote by $J_{\lambda}(t)$ the Bessel function of degree
$\lambda$ and by
$j_{\lambda}(t)=2^{\lambda}\Gamma(\lambda+1)t^{-\lambda}J_{\lambda}(t)$
the normalized Bessel function. Setting
\[
b_{\lambda}=\left(\int_{0}^{\infty}e^{-t^{2}/2}t^{2\lambda+1}\,dt\right)^{-1}=
\frac{1}{2^{\lambda}\Gamma(\lambda+1)}
\] and
$d\nu_{\lambda}(r)=b_{\lambda}r^{2\lambda+1}\,dr$, we define
$\|f\|_{2,d\nu_{\lambda}}=\left(\int_{\mathbb{R}_{+}}|f(r)|^{2}\,d\nu_{\lambda}(r)\right)^{1/2}$.

The Hankel transform is defined by
\[
\calH_{\lambda}(f)(\rho)=\int_{\mathbb{R}_{+}}f(r)j_{\lambda}(\rho
r)\,d\nu_{\lambda}(r).
\]
Note that $\calH_{\lambda}^{-1}=\calH_{\lambda}$ (see~\cite{Car,Col,Erd}).
Pitt's inequality for the Hankel transform is written as
\begin{equation}\label{eq1.4}
\|\rho^{-\beta}\calH_{\lambda}(f)(\rho)\|_{2,d\nu_\lambda}\le
c(\beta,\lambda)\|r^{\beta}f(r)\|_{2,d\nu_\lambda},\quad f\in
\mathcal{S}(\mathbb{R}_{+}),
\end{equation}
where $c(\beta,\lambda)$ is the sharp constant in (\ref{eq1.4}) and
$\mathcal{S}(\mathbb{R}_{+})$ is the the Schwartz space on~$\mathbb{R}_{+}$.
Note that if $f\in \mathcal{R}_{0}^{d}$, a study of the Hankel transform is of
special interest since the Fourier transform of a radial function can be
written as the Hankel transform.

L.~De~Carli~\cite{Car} proved that $c(\beta,\lambda)$ is finite only if $0\le
\beta<\lambda+1$. In \cite{Car,Ca} it was also studied the sharp constant in
$(L^{p}, L^{q})$ Pitt's inequality for the Hankel transform of type
\eqref{eq1.4} in the case of $1\le p\le q\le \infty$.

For $\lambda=d/2-1$, $d\in\mathbb{N}$, the constant $c(\beta,\lambda)$ was
calculated by D.~Yafaev~\cite{Yaf1}, and in the general case by
S.~Omri~\cite{Omr}. The proof of Pitt's inequality in~\cite{Omr} is rather
technical and uses the Stein--Weiss type estimate for the so-called B-Riesz
potential operator. Following~\cite{Yaf1}, we give a direct and simple proof of
inequality \eqref{eq1.4} in Section 2.

Let $|k|=\sum_{a\in R_{+}}k(a)$ and $\lambda_{k}= d/2-1+|k|$. For a radial
function $f(r)$, $r=|x|$, Pitt's inequality for the Dunkl transform
\eqref{eq1.3} corresponds to Pitt's inequality for the Hankel transform
\eqref{eq1.4} with $\lambda=\lambda_{k}$. Therefore the condition
\begin{equation}\label{eq1.5}
0\le \beta<\lambda_{k}+1
\end{equation}
is necessary for $C(\beta,k)<\infty$. Our goal is to show that in fact
$C(\beta,k)=c(\beta,\lambda_{k})$ if condition \eqref{eq1.5} holds. This will
be proved in Section 3. Moreover, in Section 4, we use Pitt's inequality for
the Dunkl transform to obtain the logarithmic uncertainty principle
(see~\cite{Bec,Omr,Sol1}). It is worth mentioning that different uncertainty
principles for the Dunkl transform have been recently studied in, e.g.,
\cite{gh,ka,rosler,sh}.

Note that for the one-dimensional Dunkl weight
\[
v_{\lambda}(t)=|t|^{2\lambda+1}, \quad
d\mu_{\lambda}(t)=\frac{v_{\lambda}(t)\,dt}{2^{\lambda+1}\Gamma(\lambda+1)},
\quad \lambda\ge -1/2,
\]
and the corresponding Dunkl transform
\[
\calF_{\lambda}(f)(s)=
\int_{\mathbb{R}}f(t)\overline{e_{\lambda}(st)}\,|t|^{2\lambda+1}\,d\mu_{\lambda}(t),\quad
e_{\lambda}(t)=j_{\lambda}(t)-ij_{\lambda}'(t),
\]
F.~Soltani~\cite{Sol1}
proved Pitt's inequality that can be equivalently written as
 \begin{equation}\label{eq1.6}
\||s|^{-\beta}\calF_{\lambda}(f)(s)\|_{2,d\mu_\lambda}\le \max
\left\{c(\beta,\lambda),c(\beta,\lambda+1)\right\}\||t|^{\beta}f(t)\|_{2,d\mu_\lambda}
\end{equation}
for $f\in \mathcal{S}(\mathbb{R})$ and $0\le \beta<\lambda+1$. Since
$c(\beta,\lambda)\ge c(\beta,\lambda+1)$ (see~\cite{Yaf1}), then in fact
\eqref{eq1.6} holds with the constant $c(\beta,\lambda)$ and therefore, we have
in this case $C(\beta,k)=c(\beta,\lambda_{k})$.

Finally, we remark that Pitt's inequality in $L^{2}$ for the multi-dimensional
Dunkl transform has been recently established in \cite{Sol2} in the case of
$\lambda_{k}-1/2<\beta<\lambda_{k}+1$. The obtained constant is not sharp.

\section{Pitt's inequality for Hankel transform}\label{s2}

\begin{theorem}\label{t1}
Let $\lambda\ge -1/2$ and $0\le\beta<\lambda+1$, then for $f\in
\mathcal{S}(\mathbb{R}_{+})$ inequality \eqref{eq1.4} holds with the sharp
constant
\[
c(\beta,\lambda)=2^{-\beta}\,\frac{\Gamma(\frac{1}{2}(\lambda+1-\beta))}{\Gamma(\frac{1}{2}(\lambda+1+\beta))}.
\]
\end{theorem}

\begin{proof}
For $\beta=0$ we have $c(\beta,\lambda)=1$ and \eqref{eq1.4} becomes Plancherel's identity
\[
\|\calH_{\lambda}(f)(\rho)\|_{2,d\nu_{\lambda}}=\|f(r)\|_{2,d\nu_{\lambda}}.
\]

Let $\beta>0$. Setting $g(r)=f(r)r^{\beta+\lambda+1/2}$
in~\eqref{eq1.4}, we arrive at
\[
\left(\int_{0}^{\infty}\left|\int_{0}^{\infty}g(r)J_{\lambda}(\rho r)(\rho
r)^{1/2-\beta}\,dr\right|^{2}\,d\rho\right)^{1/2}\le
c(\beta,\lambda)\left(\int_{0}^{\infty}|g(r)|^{2}\,dr\right)^{1/2}.
\]
Hence, $c(\beta,\lambda)$ is the norm of the integral operator
\[
A_{\lambda}g(\rho)=\int_{0}^{\infty}a_{\lambda}(\rho r)g(r)\,dr\colon
L^{2}(\mathbb{R}_{+})\to L^{2}(\mathbb{R}_{+}),
\]
where $a_{\lambda}(t)=J_{\lambda}(t)t^{1/2-\beta}$.

By \cite[Sect.~7.7]{BatErd}, the Mellin transform of the function
$a_{\lambda}(\cdot)$ for $0<\beta<\lambda+1$ and $\eta\in \mathbb{R}$ is given
by
\begin{align*}
Ma_{\lambda}(\eta)=\int_{0}^{\infty}a(t)t^{-1/2-i\eta}\,dt&=\int_{0}^{\infty}J_{\lambda}(t)t^{-\beta-i\eta}\,dt
\\
&=2^{-\beta-i\eta}\,\frac{\Gamma(\frac{1}{2}(\lambda+1-\beta-i\eta))}{\Gamma(\frac{1}{2}(\lambda+1+\beta+i\eta))}.
\end{align*}
Basic properties of the gamma function imply that $Ma_{\lambda}(\eta)$ is continuous and
\[
Ma_{\lambda}(0)=2^{-\beta}\,\frac{\Gamma(\frac{1}{2}(\lambda+1-\beta))}{\Gamma(\frac{1}{2}(\lambda+1+\beta))},\qquad
|Ma_{\lambda}(\eta)|\sim 2^{-\beta}|\eta|^{-\beta}\to 0,\quad \eta\to \infty.
\]
In \cite{Yaf1,Yaf2} it was proved that
\[
\|A_{\lambda}\|=\max_{\mathbb{R}}|Ma_{\lambda}(\eta)|=Ma_{\lambda}(0)
\]
and that the norm is not attained; that is to say, there is no function $g$
such that $\|A_{\lambda}\|=\|A_{\lambda}g\|_{L^{2}(\mathbb{R}_{+})}$ with
$\|g\|_{ L^{2}(\mathbb{R}_{+})}=1$.
\end{proof}

\begin{corollary}\label{c1}
If $\lambda_{k}= d/2-1+|k|$ and $0\le\beta<\lambda_{k}+1$, then for $f\in
\mathcal{S}(\mathbb{R}^{d})\cap \mathcal{R}_{0}^{d}$ Pitt's inequlity for the
Dunkl transform \eqref{eq1.3} holds with sharp constant
 $c(\beta,\lambda_{k})$.
\end{corollary}

\vskip 0.5cm
\section{Pitt's inequality for Dunkl transform}\label{s3}

Let $\S^{d-1}$ be the unit sphere in $\mathbb{R}^{d}$, $x'\in \S^{d-1}$, and
$dx'$ be the Lebesgue measure on the sphere. Set
$a_{k}^{-1}=\int_{\S^{d-1}}v_{k}(x')\,dx'$,
$d\omega_{k}(x')=a_{k}v_{k}(x')\,dx'$, and
$\|f\|_{2,d\omega_{k}}=\left(\int_{\S^{d-1}}|f(x')|^{2}\,d\omega_{k}(x')\right)^{1/2}.$
We have
\begin{equation}\label{eq3.1}
c_{k}^{-1}=\int_{\mathbb{R}^{d}}e^{-|x|^{2}/2}v_{k}(x)\,dx=
\int_{0}^{\infty}e^{-r^{2}/2}r^{d-1+2|k|}\,dr\int_{\S^{d-1}}v_{k}(x')\,dx'=
b_{\lambda_{k}}^{-1}a_{k}^{-1}.
\end{equation}

Let us denote by $\mathcal{H}_{n}^{d}(v_{k})$ the subspace of $k$-spherical
harmonics of degree $n\in \mathbb{Z}_{+}$ in $L^{2} (\S^{d-1},d\omega_{k})$
(see \cite[Chap.~5]{DunkXu}). Let $\Delta_{k}=\sum_{j=1}^{d}D_{j}^{2}$ be the
Dunkl Laplacian and $\mathcal{P}_{n}^{d}$ be the space of homogeneous
polynomials of degree $n$ in $\mathbb{R}^{d}$. Then
$\mathcal{H}_{n}^{d}(v_{k})$ is the restriction of $\ker \Delta_{k}\cap
\mathcal{P}_{n}^{d}$ to the sphere $\S^{d-1}$.

If $l_{n}$ is the dimension of
$\mathcal{H}_{n}^{d}(v_{k})$,
we denote by
 $\{Y_{n}^{j}\colon j=1,\ldots,l_{n}\}$ the real-valued orthonormal basis
 $\mathcal{H}_{n}^{d}(v_{k})$ in $L^{2}(\S^{d-1},d\omega_{k})$.
 A union of these bases forms an orthonormal basis in
$L^{2}(\S^{d-1},d\omega_{k})$ consisting of
$k$-spherical harmonics, i.e., we have
\begin{equation}\label{eq3.2}
L^{2}(\S^{d-1},d\omega_{k})=\sum_{n=0}^{\infty}\oplus
\mathcal{H}_{n}^{d}(v_{k}).
\end{equation}

Using \eqref{eq3.2} and the following Funk-Hecke formula for $k$-spherical harmonic
 $Y\in \mathcal{H}_{n}^{d}(v_{k})$ (see \cite{Xu})
\begin{equation}\label{eq3.3}
\int_{\S^{d-1}}Y(y')\overline{e_{k}(x,y')}\,d\omega_{k}(y')=
\frac{(-i)^{n}\Gamma(\lambda_{k}+1)}{2^{n}\Gamma(n+\lambda_{k}+1)}\,Y(x')r^{n}
j_{n+\lambda_{k}}(r),\quad x=rx'\in \mathbb{R}^{d},
\end{equation}
 similarly to \eqref{eq1.2} we have
 the
 direct sum decomposition of $L^{2}(\mathbb{R}^{d},d\mu_{k})$:
\begin{equation}\label{eq3.4}
L^{2}(\mathbb{R}^{d},d\mu_{k})=\sum_{n=0}^{\infty}\oplus
\mathcal{R}_{n}^{d}(v_{k}),\quad
\mathcal{R}_{n}^{d}(v_{k})=\mathcal{R}_{0}^{d}\otimes
\mathcal{H}_{n}^{d}(v_{k}),
\end{equation}
and that the space $\mathcal{R}_{n}^{d}(v_{k})$ is invariant under the Dunkl
transform. An example of the orthogonal basis in
$L^{2}(\mathbb{R}^{d},d\mu_{k})$ consisting of eigenfunctions of the Dunkl
transform was constructed in \cite{DunkXu}.

The next result provides a sharp constant in the Pitt inequality for the Dunkl
transform \eqref{eq1.3}.
\begin{theorem}\label{t2}
Let $\lambda_{k}= d/2-1+|k|$ and $0\le\beta<\lambda_{k}+1$, then for $f\in
\mathcal{S}(\mathbb{R}^{d})$ we~have
\[
C(\beta,k)=2^{-\beta}\,\frac{\Gamma(\frac{1}{2}(\lambda_{k}+1-\beta))}{\Gamma(\frac{1}{2}(\lambda_{k}+1+\beta))}.
\]
Sharpness of this inequality can be seen by considering radial functions.
\end{theorem}

\begin{proof}
For $\beta=0$ we have $C(\beta,k)=1$ and Pitt's inequality
\eqref{eq1.3} becomes Plancherel's identity
\[
\|\calF_{k}(f)(y)\|_{2,d\mu_{k}}=\|f(x)\|_{2,d\mu_{k}}.
\]

Let $0<\beta<\lambda_{k}+1$. If $f\in \mathcal{S}(\mathbb{R}^{d})$, then
\begin{align*}
&f_{nj}(r)=\int_{\S^{d-1}}f(rx')Y_{n}^{j}(x')\,d\omega_{k}(x')\in \mathcal{S}(\mathbb{R}_{+}),
\\
&f(rx')=\sum_{n=0}^{\infty}\sum_{j=1}^{l_{n}}f_{nj}(r)Y_{n}^{j}(x'),
\\
&\int_{\S^{d-1}}|f(rx')|^{2}\,d\omega_{k}(x')=\sum_{n=0}^{\infty}\sum_{j=1}^{l_{n}}|f_{nj}(r)|^{2}.
\end{align*}
Using spherical coordinates, decomposition of $L^{2}(\mathbb{R}^{d},d\mu_{k})$
\eqref{eq3.4}, formulas \eqref{eq3.1} and \eqref{eq3.3}, and the property
$e_{k}(tx,y)=e_{k}(ty,x)$, we get that
\begin{equation}\label{eq3.5}
\begin{aligned}[t]
\int_{\mathbb{R}^{d}}|x|^{2\beta}|f(x)|^{2}\,d\mu_{k}(x)
&=b_{\lambda_{k}}\int_{0}^{\infty}r^{2\beta+d-1+2|k|}
\int_{\S^{d-1}}|f(rx')|^{2}\,d\omega_{k}(x')\,dr
\\
&=b_{\lambda_{k}}\int_{0}^{\infty}r^{2\beta+d-1+2|k|}
\sum_{n=0}^{\infty}\sum_{j=1}^{l_{n}}|f_{nj}(r)|^{2}\,dr
\\
&=\sum_{n=0}^{\infty}\sum_{j=1}^{l_{n}}\int_{0}^{\infty}|f_{nj}(r)|^{2}r^{2\beta}\,d\nu_{\lambda_{k}}(r),
\end{aligned}
\end{equation}
\begin{equation*}
\begin{aligned}
\calF_{k}(f)(y)&=\int_{\mathbb{R}^{d}}f(x)\overline{e_{k}(x,
y)}\,d\mu_{k}(x)= b_{\lambda_{k}}\int_{0}^{\infty}r^{d-1+2|k|}
\int_{\S^{d-1}}f(rx')\,d\omega_{k}(x')\,dr
\\
&=\sum_{n=0}^{\infty}\sum_{j=1}^{l_{n}}b_{\lambda_{k}}\int_{0}^{\infty}f_{nj}(r)r^{2\lambda_{k}+1}
\int_{\S^{d-1}}Y_{n}^{j}(x')\overline{e_{k}(rx',\rho y')}\,d\omega_{k}(x')\,dr
\\
&=\sum_{n=0}^{\infty}\sum_{j=1}^{l_{n}}\frac{(-i)^{n}\Gamma(\lambda_{k}+1)}{2^{n}\Gamma(n+\lambda_{k}+1)}\,Y_{n}^{j}(y')
\int_{0}^{\infty}f_{nj}(r)j_{n+\lambda_{k}}(\rho r)(\rho r)^{n}\,d\nu_{\lambda_{k}}(r),
\end{aligned}
\end{equation*}
and
\begin{multline}\label{eq3.6}
\int_{\mathbb{R}^{d}}|y|^{-2\beta}|\calF_{k}(f)(y)|^{2}\,d\mu_{k}(y)=
\sum_{n=0}^{\infty}\sum_{j=1}^{l_{n}}\frac{\Gamma^{2}(\lambda_{k}+1)}{2^{2n}\Gamma^{2}(n+\lambda_{k}+1)}
\\
\times
\int_{0}^{\infty}\left|\int_{0}^{\infty}f_{nj}(r)j_{n+\lambda_{k}}(\rho r)(\rho
r)^{n}\,d\nu_{\lambda_{k}}(r)\right|^{2}\rho^{-2\beta}\,d\nu_{\lambda_{k}}(\rho).
\end{multline}
Suppose that
$g\in \mathcal{S}(\mathbb{R}_{+})$, $n\in \mathbb{Z}_{+}$, and
$0<\beta<\lambda_{k}+1+n$. Let us show that
\begin{multline}\label{eq3.7}
\left\|\rho^{-\beta}\int_{0}^{\infty}g(r)j_{n+\lambda_{k}}(\rho r)(\rho r)^{n}\,d\nu_{\lambda_{k}}(r)\right\|_{2,d\nu_{\lambda_{k}}}
\\
\le \frac{2^{n}\Gamma(n+\lambda_{k}+1)c(\beta,\lambda_{k}+n)}{\Gamma(\lambda_{k}+1)}\,\|r^{\beta}g(r)\|_{2,d\nu_{\lambda_{k}}},
\end{multline}
where $c(\beta,\lambda_{k}+n)$ is given in Theorem \ref{t1} with $\lambda=\lambda_{k}+n$.
Setting in \eqref{eq3.7} $g(r)=u(r)r^{n}$, we rewrite it as follows:
\begin{equation}\label{eq3.8}
\left\|\rho^{-\beta}\int_{0}^{\infty}u(r)j_{n+\lambda_{k}}\,d\nu_{\lambda_{k}+n}(r)\right\|_{2,d\nu_{\lambda_{k}+n}}
\le c(\beta,\lambda_{k}+n)\|r^{\beta}u(r)\|_{2,d\nu_{\lambda_{k}+n}},
\end{equation}
which is \eqref{eq1.4} with $\lambda=\lambda_{k}+n$.

Since $c(\beta,\lambda_{k}+n)$ is decreasing with $n$ (see \cite{Yaf1}), then
using \eqref{eq3.5}, \eqref{eq3.6}, and \eqref{eq3.7}, we arrive at
\begin{align*}
\int_{\mathbb{R}^{d}}|y|^{-2\beta}|\calF_{k}(f)(y)|^{2}\,d\mu_{k}(y)
&\le
\sum_{n=0}^{\infty}\sum_{j=1}^{l_{n}}c^{2}(\beta,\lambda_{k}+n)\int_{0}^{\infty}|f_{nj}(r)|^{2}r^{2\beta}\,d\nu_{\lambda_{k}}(r)
\\
&\le c^{2}(\beta,\lambda_{k})\sum_{n=0}^{\infty}\sum_{j=1}^{l_{n}}\int_{0}^{\infty}|f_{nj}(r)|^{2}r^{2\beta}\,d\nu_{\lambda_{k}}(r)
\\
&=c^{2}(\beta,\lambda_{k})
\int_{\mathbb{R}^{d}}|x|^{2\beta}|f(x)|^{2}\,d\mu_{k}(x).
\end{align*}
\end{proof}

In the proof of Theorem~\ref{t2} we obtained the following result (see
\eqref{eq3.8}).

\begin{corollary}\label{c2}
Let $n\in \mathbb{N}$, $\lambda_{k}= d/2-1+|k|$, and
$0\le\beta<\lambda_{k}+1+n$, then for $f\in \mathcal{S}(\mathbb{R}^{d})\cap
\mathcal{R}_{n}^{d}(v_{k})$ we have Pitt's inequality for the Dunkl
transform~\eqref{eq1.3} with sharp constant $c(\beta,\lambda_{k}+n)$.
\end{corollary}
For the Fourier transform Corollary~\ref{c2} was established in~\cite{Yaf1}.

\vskip 0.5cm
\section{ Logarithmic uncertainty principle for
Dunkl transform}\label{s4}

W.~Beckner in \cite{Bec} proved the logarithmic uncertainty principle for the Fourier transform using Pitt's inequality~\eqref{eq1.1}:
if $f\in \mathcal{S}(\mathbb{R}^{d})$, then
\[
\int_{\mathbb{R}^{d}}\ln(|x|)|f(x)|^{2}\,dx+
\int_{\mathbb{R}^{d}}\ln(|y|)|\widehat{f}(y)|^{2}\,dy\ge
\left(\psi \Bigl(\frac{d}{4}\Bigr)+\ln
2\right)\int_{\mathbb{R}^{d}}|f(x)|^{2}\,dx,
\]
where $\psi(t)=d\ln \Gamma(t)/dt$ is the $\psi$-function.

For the Hankel transform the logarithmic uncertainty principle reads as follows (see \cite{Omr}): if
 $f\in \mathcal{S}(\mathbb{R}_+)$ and $\lambda\ge -1/2$, then
\begin{multline}\label{eq4.1}
\int_{\mathbb{R}_+}\ln(t)|f(t)|^{2}t^{2\lambda+1}\,dt+\int_{\mathbb{R}_+}
\ln(s)|\calH_{\lambda}(f)(s)|^{2}s^{2\lambda+1}\,ds
\\
\ge
\left(\psi \Bigl(\frac{\lambda+1}{2}\Bigr)+\ln
2\right)\int_{\mathbb{R}_+}|f(t)|^{2}t^{2\lambda+1}\,dt.
\end{multline}
For the one-dimensional Dunkl transform of functions $f\in \mathcal{S}(\mathbb{R})$, F.~Soltani~\cite{Sol1} has recently proved that
\begin{multline*}
\int_{\mathbb{R}}\ln(|t|)|f(t)|^{2}|t|^{2\lambda+1}\,dt+
\int_{\mathbb{R}}\ln(|s|)|\calF_{\lambda}(f)(s)|^{2}|s|^{2\lambda+1}\,ds
\\
\ge \left(\min \left\{\psi \Bigl(\frac{\lambda+1}{2}\Bigr),\psi
\Bigl(\frac{\lambda+2}{2}\Bigr)\right\} +\ln
2\right)\int_{\mathbb{R}}|f(t)|^{2}|t|^{2\lambda+1}\,dt.
\end{multline*}
Since $\psi$ is increasing the latter can be written as
\begin{multline}\label{eq4.2}
\int_{\mathbb{R}}\ln(|t|)|f(t)|^{2}|t|^{2\lambda+1}\,dt+
\int_{\mathbb{R}}\ln(|s|)|\calF_{\lambda}(f)(s)|^{2}|s|^{2\lambda+1}\,ds
\\
\ge \left(\psi \Bigl(\frac{\lambda+1}{2}\Bigr)+
\ln 2\right)\int_{\mathbb{R}}|f(t)|^{2}|t|^{2\lambda+1}\,dt,
\end{multline}
which is the logarithmic uncertainty principle for the one-dimensional Dunkl
transform.

Using Pitt's inequality \eqref{eq1.3} we obtain the logarithmic uncertainty
principle the multi-dimensional Dunkl transform.

\begin{theorem}\label{t3}
Let $\lambda_{k}= d/2-1+|k|$ and $f\in \mathcal{S}(\mathbb{R}^{d})$.
We have
\begin{multline}\label{eq4.3}
\int_{\mathbb{R}^{d}}\ln(|x|)|f(x)|^{2}\,d\mu_{k}(x)+\int_{\mathbb{R}^{d}}\ln(|y|)|\calF_{k}(f)(y)|^{2}\,d\mu_{k}(y)
\\
\ge \left(\psi \Bigl(\frac{\lambda_{k}+1}{2}\Bigr)+\ln
2\right)\int_{\mathbb{R}^{d}}|f(x)|^{2}\,d\mu_{k}(x).
\end{multline}
\end{theorem}
\begin{proof}
We write the Pitt inequality~\eqref{eq1.3} in the following form
 \[
\int_{\mathbb{R}^{d}}|y|^{-\beta}|\calF_{k}(f)(y)|^{2}\,d\mu_{k}(y)\le
c^{2}(\beta/2,
\lambda_{k})\int_{\mathbb{R}^{d}}|x|^{\beta}|f(x)|^{2}\,d\mu_{k}(x),\quad 0\le
\beta<2(\lambda_{k}+1).
\]
For $\beta\in \left(-2(\lambda_{k}+1),2(\lambda_{k}+1)\right)$ define the function
\[
\varphi
(\beta)=\int_{\mathbb{R}^{d}}|y|^{-\beta}|\calF_{k}(f)(y)|^{2}\,d\mu_{k}(y)-
c^{2}(\beta/2,\lambda_{k})\int_{\mathbb{R}^{d}}|x|^{\beta}|f(x)|^{2}\,d\mu_{k}(x).
\]
Since $|\beta|<2(\lambda_{k}+1)$ and $f,\,\calF_{k}(f)\in
\mathcal{S}(\mathbb{R}^{d})$, then
\[
\int_{|x|\le 1}|\!\ln(|x|)||x|^{\beta}v_{k}(x)\,dx=
\int_{0}^{1}|\!\ln(r)|r^{\beta+2\lambda_{k}+1}\,dr\int_{\S^{d-1}}v_{k}(x')\,dx'<\infty,
\]
which gives
\[
|y|^{-\beta}\ln(|y|)|\calF_{k}(f)(y)|^{2}v_{k}(y)\in L^{1}(\mathbb{R}^{d})\quad
\text{and}\quad \ln(|x|)|x|^{\beta}|f(x)|^{2}v_{k}(x)\in L^{1}(\mathbb{R}^{d}).
\]
Therefore,
\begin{multline}\label{eq4.4}
\varphi'(\beta)=-\int_{\mathbb{R}^{d}}|y|^{-\beta}\ln(|y|)|\calF_{k}(f)(y)|^{2}\,d\mu_{k}(y)
\\
-c^{2}(\beta/2,\lambda_{k})\int_{\mathbb{R}^{d}}|x|^{\beta}\ln(|x|)|f(x)|^{2}\,d\mu_{k}(x)
\\
-\frac{dc^{2}(\beta/2,\lambda_{k})}{d\beta}\int_{\mathbb{R}^{d}}|x|^{\beta}|f(x)|^{2}\,d\mu_{k}(x).
\end{multline}
Pitt's inequality and Plancherel's theorem imply that $\varphi (\beta)\le 0$ for $\beta>0$ and
$\varphi (0)=0$ correspondingly, hence
\[
\varphi'(0_{+})=\lim_{\beta\to 0_{+}}\frac{\varphi
(\beta)-\varphi(0)}{\beta}\le 0.
\]
Noting that
\begin{equation}\label{eq4.5}
-\frac{dc^{2}(\beta/2,\lambda_{k})}{d\beta}\biggr|_{\beta=0}=
\psi \Bigl(\frac{\lambda_{k}+1}{2}\Bigr)+\ln 2,
\end{equation}
we conclude that proof of \eqref{eq4.3} combining
 \eqref{eq4.4} and \eqref{eq4.5}.
\end{proof}

\smallbreak
For the radial functions inequality \eqref{eq4.3} with $\lambda_{k}=\lambda$
implies inequality~\eqref{eq4.1}.

\end{document}